\documentclass[leqno,11pt]{amsart}
\usepackage[bookmarksnumbered]{hyperref}	%Load hyperref before babel so that renewcommand for sectionautorefname works!
\usepackage[ngerman,USenglish]{babel}		%for nicer umlaut dots in Universit\"at zu K\"oln
\usepackage{color}												%For comments
\usepackage[ascii]{inputenc}							%Ensure source contains only ASCII characters
\usepackage{aliascnt}										%Make autoref work with joint numbering
\usepackage{amssymb}											%for \lesssim
\usepackage{microtype}										%Verbesserte Umbrueche
\usepackage[all]{hypcap}

\renewcommand{\mathbb}{\mathbf}
\renewcommand{\partial}{d}
\renewcommand{\epsilon}{\varepsilon}

\newcommand{\newjointcountertheorem}[3]{
	\newaliascnt{#1}{#2}
	\newtheorem{#1}[#1]{#3}
	\aliascntresetthe{#1}	
}

\newtheorem*{thm-gauss}{Theorem \ref{gauss}}
\newtheorem*{thm-macmahon}{Theorem \ref{macmahon}}
\newtheorem*{cor-codes}{Corollary \ref{binary codes and matroids}}

\newtheorem{thm}{Theorem}[section]
\newjointcountertheorem{satz}{thm}{Satz}
\newjointcountertheorem{lem}{thm}{Lemma}
\newjointcountertheorem{cor}{thm}{Corollary}
\newjointcountertheorem{prp}{thm}{Proposition}
\newjointcountertheorem{cnj}{thm}{Conjecture}
\newjointcountertheorem{que}{thm}{Question}
\newjointcountertheorem{fct}{thm}{Fact}
\theoremstyle{definition}
\newjointcountertheorem{dfn}{thm}{Definition}
\newjointcountertheorem{ntn}{thm}{Notation}
\newjointcountertheorem{rem}{thm}{Remark}
\newjointcountertheorem{nte}{thm}{Note}
\newjointcountertheorem{exl}{thm}{Example}

\DeclareMathOperator{\Cov}{Cov}

\DeclareMathOperator{\des}{des}
\DeclareMathOperator{\inv}{inv}
\DeclareMathOperator{\Var}{Var}

\newcommand{\central}{\mathrm{c}}
\newcommand{\LandauO}{O}
\newcommand{\Landauo}{o}

\newcommand{\qbinom}[2]{\genfrac{[}{]}{0pt}{}{#1}{#2}}
\newcommand{\qmultinom}[2]{\genfrac{[}{]}{0pt}{}{#1}{#2}}
\newcommand{\charac}[1]{\mathrm{ch}(#1)}
\newcommand{\annrel}[2]{\stackrel{\makebox[0pt]{\scriptsize #1}}{#2}} % A nicely formatted text over =

\def\Snospace~{\S{}}

\newcommand{\E}[2]{\operatorname{E}_{#1}[#2]}

\numberwithin{equation}{section}
\allowdisplaybreaks[4] % To make sequences of aligned breakable

\hypersetup{pdfauthor={Thomas Bliem and Stavros Kousidis},pdftitle={The number of flags in finite vector spaces: Asymptotic normality and Mahonian statistics}}

\begin{document}

	\title[Number of flags in finite vector spaces]{The number of flags in finite vector spaces: Asymptotic normality and Mahonian statistics}
	
	\author[Thomas Bliem]{Thomas Bliem}
  	\address{\foreignlanguage{ngerman}{Thomas Bliem\\Kempener Str. 57\\50733 K\"oln\\Germany}}
  	\email{\href{mailto:tbliem@gmx.de}{tbliem@gmx.de}}

  	\author[Stavros Kousidis]{Stavros Kousidis}
  	\address{Stavros Kousidis, Institute for Theoretical Physics, ETH Z\"urich, Wolfgang--Pauli--Strasse 27, CH-8093 Z\"urich, Switzerland\newline \indent Institute of Physics, University of Freiburg, Rheinstrasse 10, 79104 Freiburg, Germany}
	\email{\href{mailto:st.kousidis@googlemail.com}{st.kousidis@googlemail.com}}

	\subjclass[2010]{05A16 Asymptotic enumeration, 60B99 Probability theory on algebraic and topological structures, 94B05 Linear codes, 06B15 Representation theory.}
	\keywords{Galois number, Gaussian normal distribution, MacMahon inversion statistic, Rogers--Szeg\H{o} polynomial, linear code, Demazure module, symmetric group, descent-inversion statistic}
	
	\begin{abstract}
		We study the generalized Galois numbers which count flags of length $r$ in $N$-dimensional vector spaces over finite fields. We prove that the coefficients of those polynomials are asymptotically Gaussian normally distributed as $N$ becomes large. Furthermore, we interpret the generalized Galois numbers as weighted inversion statistics on the descent classes of the symmetric group on $N$ elements and identify their asymptotic limit as the Mahonian inversion statistic when $r$ approaches $\infty$. Finally, we apply our statements to derive further statistical aspects of generalized Rogers--Szeg\H{o} polynomials, re-interpret the asymptotic behavior of linear $q$-ary codes and characters of the symmetric group acting on subspaces over finite fields, and discuss implications for affine Demazure modules and joint probability generating functions of descent-inversion statistics.
	\end{abstract}

	\maketitle
	\tableofcontents

	\section{Introduction}
			
	Let $G_N^{(r)} (q)$ denote the number of flags $0 = V_0 \subseteq \cdots \subseteq V_r = \mathbf{F}_q^N$ of length $r$ in an $N$-dimensional vector space over a field with $q$ elements where repetitions are allowed.
	Then $G_N^{(r)} (q)$ is a polynomial in $q$, a so-called generalized Galois number \cite{v10}.
	In particular, when $r=2$ these are the Galois numbers which give the total number of subspaces of an $N$-dimensional vector space over a finite field \cite{MR0252232}.
		
		The generalized Galois numbers are a biparametric family of polynomials, each with non-negative integral coefficients, in the parameters $N$ and $r$.
		We want to analyze their limiting properties as those parameters become large.
		Viewing each generalized Galois number as a discrete distribution on the real line, we determine the asymptotic behavior of this biparametric family of distributions. Our first result will be:
		
		\begin{thm-gauss}
			For $r\geq2$ and $N \in \mathbf{N}$ let $G_{N,r}$ be a random variable with probability generating function $\E{}{q^{G_{N,r}}} = r^{-N} \cdot G_N^{(r)} (q)$.
			Then,
			\begin{align*}
				\E{}{G_{N,r}} & = \frac{r-1}{4r} N(N-1) , \\ 
				\Var(G_{N,r}) & = \frac{(r-1)(r+1)}{72 r^2} N(N-1)(2N+5) .
			\end{align*}
			For fixed $r$ and $N \rightarrow \infty$, the distribution of the random variable
			\[
				\frac{G_{N,r} - \E{}{G_{N,r}}}{\Var(G_{N,r})^{\frac 12}}
			\]
			converges weakly to the standard normal distribution.
		\end{thm-gauss}
		
		Furthermore, we derive an exact formula in terms of weighted inversion statistics on the descent classes of the symmetric group and derive the asymptotic behavior with respect to the second parameter:
		
		\begin{thm-macmahon}
			Consider the Galois number $G_N^{(r)}(q) \in \mathbf{N}[q]$ for $r\geq2$ and $N \in \mathbf{N}$. Let $\mathfrak{S}_N$ be the symmetric group on $N$ elements, and for a permutation $\pi \in \mathfrak{S}_N$ denote by $\inv(\pi)$ its number of inversions and by $\des(\pi)$ the cardinality of its descent set $D(\pi)$. Then,
			\[
				G_N^{(r)}(q) = \sum_{\pi \in \mathfrak{S}_N}\binom{N+r-1-\des(\pi)}{N}  q^{\inv(\pi)}
				.
			\]
			For fixed $N$ and $r \rightarrow \infty$ we have
			\[
				\frac{N !}{r^N} \cdot G_N^{(r)}(q) \to \sum_{\pi \in \mathfrak{S}_N} q^{\inv(\pi)}  = [N]_q!
				.
			\]			
		\end{thm-macmahon}
		
		To state one application, our exact formula in \autoref{macmahon} allows us to re-interpret the asymptotic behavior of the numbers of equivalence classes of linear $q$-ary codes \cite{MR2177491,MR2492098,MR1755766, MR2191288} under permutation equivalence $(\mathfrak{S})$, monomial equivalence $(\mathfrak{M})$ and semi-linear monomial equivalence $(\Gamma)$ as follows:
		
		\begin{cor-codes}
			The number of linear $q$-ary codes of length $n$ up to equivalence $(\mathfrak{S})$, $(\mathfrak{M})$ and $(\Gamma)$ is given asymptotically, as $q$ is fixed and $n \rightarrow \infty$, by
			 \begin{align}
			 	N_{n,q}^{\mathfrak{S}} & \sim \frac{1}{n!} \sum_{\substack{\pi \in \mathfrak{S}_n \\ \des(\pi) \leq 1}}  \binom{n+1-\des(\pi)}{n} q^{\inv(\pi)} , \\
				N_{n,q}^{\mathfrak{M}} & \sim \frac{1}{n!(q-1)^{n-1}} \sum_{\substack{\pi \in \mathfrak{S}_n \\ \des(\pi) \leq 1}}  \binom{n+1-\des(\pi)}{n} q^{\inv(\pi)}, \\
				N_{n,q}^{\Gamma} & \sim \frac{1}{n!(q-1)^{n-1}a} \sum_{\substack{\pi \in \mathfrak{S}_n \\ \des(\pi) \leq 1}}  \binom{n+1-\des(\pi)}{n} q^{\inv(\pi)} ,
			\end{align}
			where $a = \left| \mathrm{Aut}(\mathbf{F}_q) \right| = \log_p(q)$ with $p = \mathrm{char}(\mathbf{F}_q)$.
			In particular, the numerator of the asymptotic numbers of linear $q$-ary codes is the (weighted) inversion statistic on the permutations having at most $1$ descent. 
		\end{cor-codes}
		
		The organization of our article is as follows. 
		Since the generalized Galois numbers are a specialization of the generalized Rogers--Szeg\H{o} polynomials \cite{v10}, which are generating functions of $q$-multinomial coefficients \cite{r1893a,r1893b,s1926}, we summarize the statistical behavior of the $q$-multinomial coefficients in \autoref{sec:notation-preliminaries}.
		The determination of mean and variance for generalized Galois numbers and of the higher cumulants for $q$-multinomial coefficients in \autoref{sec:gauss}, allow us to prove the asymptotic normality of the generalized Galois numbers (\autoref{gauss}) through the method of moments.
		In \autoref{sec:macmahon} we analyze the combinatorial interpretation of $q$-multinomial coefficients in terms of inversion statistics on permutations \cite{{MR1442260}}. Based on our interpretation of the generalized Galois numbers as weighted inversion statistics on the descent classes of the symmetric group, we describe their limiting behavior towards the Mahonian inversion statistic (\autoref{macmahon}). We conclude with applications of our results in \autoref{sec:applications}. That is, we derive further statistical aspects of generalized Rogers--Szeg\H{o} polynomials in \autoref{covariance Rogers--Szegoe}, re-interpret the asymptotic behavior of the numbers of linear $q$-ary codes in \autoref{binary codes and matroids}, and discuss implications for affine Demazure modules in \autoref{demazure basic specialization} and joint probability generating functions of descent-inversion statistics \eqref{deformed galois}, \eqref{genfunc deformed galois}.
	 
	\section{Notation and preliminaries}
	\label{sec:notation-preliminaries}
		
	We denote by $\mathbf{N}$ the set of nonnegative integers $\{0,1,2,3, \ldots\}$.
	Let $q$ be a variable, $N \in \mathbf{N}$ and ${\bf k} = (k_1 , \ldots , k_r) \in \mathbf{N}^r$. The $q$-multinomial coefficient is defined as
	\begin{equation}
	\label{q-multinomial}
		\qmultinom{N}{{\bf k}}_q = \begin{cases} \frac{[N]_q !}{[k_1]_q ! \ldots [k_r]_q !} & \text{if } k_1 + \cdots + k_r = N, \\ 0 & \text{otherwise} . \end{cases}
	\end{equation}
	Here, $[k]_q ! = \prod_{i=1}^k \frac{1-q^i}{1-q}$ denotes the $q$-factorial.
	Note that the $q$-multinomial coefficient is a polynomial in $q$ and
	\[
		\left. \qmultinom{N}{{\bf k}}_q \right\vert_{q=1}
		= \binom{N}{{\bf k}} = \frac{N!}{k_1 ! \cdots k_r !}.
	\]
	For $N \in \mathbf{N}$, the generalized $N$-th Rogers--Szeg\H{o} polynomial $H^{(r)}_N({\bf z}, q) \in \mathbf{C}[z_1, \ldots , z_r, q]$ is the generating function of the $q$-multinomial coefficients:
	\[
		H^{(r)}_N({\bf z} , q)
		= \sum_{{\bf k} \in \mathbf{N}^r}\qmultinom{N}{{\bf k}}_q {\bf z^k} .
	\]
	Here we use multi-exponent notation ${\bf z^k} = z_1^{k_1} \cdots z_r^{k_r}$ for  ${\bf k} = (k_1 , \ldots , k_r) \in \mathbf{N}^r$. Note that by our definition of the $q$-multinomial coefficients it is convenient to suppress the condition $k_1 + \cdots + k_r = N$ in the summation index.
	
	As described in \cite{v10}, the $q$-multinomial coefficient $\qbinom{N}{\mathbf{k}}_q$ counts the number of flags $0 = V_0 \subseteq \cdots \subseteq V_r = \mathbf{F}_q^N$ subject to the conditions $\dim(V_i) = k_1 + \cdots + k_i$, and consequently the specialization of the generalized Rogers--Szeg\H{o} polynomial $H^{(r)}_N ({\bf z},q)$ at ${\bf z} = {\bf 1} = (1, \ldots, 1)$ counts the total number of flags of subspaces of length $r$ in $\mathbf{F}_q^N$. This number (a polynomial in $q$) is called a generalized Galois number and denoted by $G_N^{(r)}(q)$. In particular, when $r=2$, the specializations of the Rogers--Szeg\H{o} polynomials are the Galois numbers $G_N(q)$ which count the number of subspaces in $\mathbf{F}_q^N$ \cite{MR0252232}.
	
	We will need notation from the context of symmetric groups. $D(\pi)$ is the descent set of $\pi$, $\mathcal{D}_T$ the descent class, $\des(\pi)$ the number of descents, and $\inv(\pi)$ the number of inversions of $\pi$, i.e.,
		\begin{align*}
			D(\pi) & = \{ i : \pi(i) > \pi(i+1) \} \\
			\mathcal{D}_T & = \{ \pi : D(\pi) = T \} \\
			\des(\pi) & = \lvert D(\pi) \rvert \\
			\inv(\pi) & = \lvert \{ (i,j) : i<j \text{ and } \pi(i) > \pi(j) \} \rvert.
		\end{align*}
	The sign $\sim$ refers to asymptotic equivalence, that is for $f,g : \mathbf{N} \rightarrow \mathbf{R}_{>0}$ we write $f(n) \sim g(n)$ if $\lim_{n \rightarrow \infty} f(n)/g(n) =1$. We write $f(n) = \LandauO(g(n))$ if there exists a constant $C>0$ such that $f(n)\leq Cg(n)$ for all sufficiently large $n$, and $f(n) = \Landauo(g(n))$ if $\lim_{n \rightarrow \infty} f(n)/g(n) =0$.
	
	Let us recollect some known results about statistics of $q$-multinomial coefficients. Note that one has the usual differentiation method.
	\begin{prp}
	\label{taylor and moments}
		Let $X$ be a discrete random variable with probability generating function $E[q^X] = f(q) \in \mathbf{R}[q]$.
		Then,
		\begin{align*}
			\E{}{X} & =\left. \frac{d}{dq} f(q) \right|_{q=1} , \\
			\Var(X) & = \left. \frac{\partial^2}{\partial q^2} q^{-\E{}{X}}f(q) \right|_{q=1} .
		\end{align*}
	\end{prp}
	
	 Via \autoref{taylor and moments} one can prove from the definition \eqref{q-multinomial} of the $q$-multinomial coefficient:
	
	\begin{prp}[\protect{\cite[Equation (1.9) and (1.10)]{cjz11}}]
	\label{moments q-multinomials}
		For ${\bf k} = (k_1 , \ldots , k_r) \in \mathbf{N}^r$ let $X_{N,{\bf k}}$ be a random variable with probability generating function $\E{}{q^{X_{N,{\bf k}}}} = \binom{N}{{\bf k}}^{-1} \cdot \qbinom{N}{{\bf k}}_q$.
		Then,
		\begin{align*}
			\E{}{X_{N,{\bf k}}} & = \frac{e_2 ({\bf k})}{2} , \\ 
			\Var(X_{N,{\bf k}}) & = \frac{(e_1({\bf k}) + 1)e_2({\bf k}) - e_3({\bf k})}{12} . 
		\end{align*}
		Here, $e_i({\bf k})$ denotes the $i$-th elementary symmetric function in the variables ${\bf k} = (k_1 , \ldots , k_r)$.
	\end{prp}
		
	\section{Asymptotic normality of Galois numbers}
	\label{sec:gauss}
	
	Let us start by computing mean and variance of the generalized Galois numbers.
	
	\begin{lem}
	\label{mean variance galois}
		Let $G_{N,r}$ be a random variable with probability generating function $\E{}{q^{G_{N,r}}} = r^{-N} \cdot G_N^{(r)} (q)$.
		Then,
			\begin{align*}
				\E{}{G_{N,r}} & = \frac{r-1}{4r} N(N-1) , \\
				\Var(G_{N,r}) & = \frac{(r-1)(r+1)}{72 r^2} N(N-1)(2N+5) .
			\end{align*}
	\end{lem}
	
	\begin{proof}
		By \autoref{taylor and moments} we have to compute the value of the derivatives $\frac{\partial}{\partial q}$ and $\frac{\partial^2}{\partial q^2}$ of $G_N^{(r)}(q) = H^{(r)}_N({\bf 1},q)$ evaluated at $q=1$. Since $H^{(r)}_N({\bf 1},q) = \sum_{{\bf k} \in \mathbf{N}^r}\qmultinom{N}{{\bf k}}_q$, they can be computed from \autoref{moments q-multinomials} via index manipulations in sums involving multinomial coefficients. We will need the identities
		\begin{equation}
			\label{multimomial sum elementary symmetric}
			\sum_{{\bf k} \in \mathbf{N}^r} \binom{N}{{\bf k}} e_s ({\bf k}) = s! \binom{N}{s}\binom{r}{s} r^{N-s} ,
		\end{equation}
		and
		\begin{equation}
		\label{multinomial sum elementary symmetric square}
			\sum_{{\bf k} \in \mathbf{N}^r} \binom{N}{{\bf k}} e_2 ({\bf k})^2 = r^N \frac{N^2 (r-1)^2 - N(r-1)^2+2(r-1)}{4 r^2} N(N-1) .
		\end{equation}
		The last identity follows from 
		\[
			 e_2 ^2 = \frac 12 (p_4 - e_1^4 +4 e_2 e_1^2 -4 e_3 e_1 + 4 e_4)
		\]
		where $p_s$ denotes the $s$-th power sum and
		\begin{align*}
			 \sum_{{\bf k} \in \mathbf{N}^r} \binom{N}{{\bf k}} p_4 ({\bf k}) & = N r^N + 14 \binom{N}{2} r^{N-1} + 36 \binom{N}{3} r^{N-2} + 24 \binom{N}{4} r^{N-3}
			 .
		\end{align*}
		Now, $G_N^{(r)}(1) = r^N$ and by \eqref{multimomial sum elementary symmetric}
		\begin{align*}
			\E{}{G_{N,r}} & = \frac{1}{r^N} \frac{\partial}{\partial q} \bigg|_{q=1} H^{(r)}_N({\bf 1}, q) \\
				& = \frac{1}{r^N}  \sum_{{\bf k} \in \mathbf{N}^r}\binom{N}{{\bf k}} \frac{e_2 ({\bf k})}{2} \\
				& =  \frac{r-1}{4r} N(N-1)
				.
		\end{align*}
		For the variance we will also need \eqref{multinomial sum elementary symmetric square}. That is,
		\begin{align*}
			\Var(G_{N,r}) & = \frac{1}{r^N} \frac{\partial^2}{(\partial q) ^2}\bigg|_{q=1} q^{-\E{}{G_{N,r}}} H^{(r)}_N({\bf 1}, q) \\
			& = \frac{1}{r^N} \sum_{{\bf k} \in \mathbf{N}^r} \frac{\partial^2}{(\partial q) ^2}\bigg|_{q=1} q^{-\E{}{G_{N,r}}} \qmultinom{N}{{\bf k}}_q \\
			& = \frac{1}{r^N} \sum_{{\bf k} \in \mathbf{N}^r} \E{}{G_{N,r}}(\E{}{G_{N,r}} + 1) \qmultinom{N}{{\bf k}}_{q=1}  \\
					& \quad \quad \quad - 2 \E{}{G_{N,r}}  \frac{\partial}{\partial q}\bigg|_{q=1}\qmultinom{N}{{\bf k}}_q + \frac{\partial^2}{(\partial q)^2}\bigg|_{q=1}\qmultinom{N}{{\bf k}}_q \\
			& = \frac{1}{r^N} \sum_{{\bf k} \in \mathbf{N}^r}	(\E{}{G_{N,r}}^2 + \E{}{G_{N,r}}) \binom{N}{{\bf k}} - 2 \E{}{G_{N,r}} \binom{N}{{\bf k}} \frac{e_2({\bf k})}{2} \\
					& \quad \quad + \binom{N}{{\bf k}} \left( \frac{(e_1({\bf k}) + 1)e_2({\bf k}) - e_3({\bf k})}{12} + \frac{e_2({\bf k})}{2} \left( \frac{e_2({\bf k})}{2} -1 \right) \right) \\
			& = \frac{1}{r^N} \sum_{{\bf k} \in \mathbf{N}^r} \binom{N}{{\bf k}} \bigg( \big( \E{}{G_{N,r}}^2 + \E{}{G_{N,r}} \big)  \\
						& \quad \quad + \big( \frac 16 (N+1) - 2 \E{}{G_{N,r}} - 1 \big) \frac{e_2({\bf k})}{2} - \frac{e_3({\bf k})}{12} + \frac{e_2({\bf k})^2 }{4} \bigg) \\
			& = \E{}{G_{N,r}}^2 + \E{}{G_{N,r}} + \big( \frac 16 (N+1) - 2 \E{}{G_{N,r}} - 1 \big)  \E{}{G_{N,r}} \\
						& \quad \quad - \frac{6}{12} \binom{r}{3} \binom{N}{3} \frac{1}{r^3} + \frac{1}{4r^N} \sum_{{\bf k} \in \mathbf{N}^r} \binom{N}{{\bf k}} e_2({\bf k})^2 \\
						& = \frac{(r-1)(r+1)}{72 r^2} N(N-1)(2N+5) .
					\qedhere
		\end{align*}
	\end{proof}
		
		In order to prove asymptotic normality, we make use of the well-known method of moments \cite[Theorem 4.5.5]{MR1796326}.
		We will need some preparatory statements. First, from the description through elementary symmetric polynomials in \autoref{moments q-multinomials} we can derive the asymptotic behavior of the first two central moments of the central $q$-multinomial coefficients and their square-root distant neighbors.
	
	\begin{prp}
	\label{asymptotic moments central q-multinomials}
		Let ${\bf k}_\central = (k_1^{(N)} , \ldots , k_r^{(N)}) \in \mathbf{N}^r$ be a sequence such that $k_1^{(N)} + \cdots + k_r^{(N)} = N$ and $\bigl\lfloor \frac{N}{r} \bigr\rfloor \leq k_i^{(N)}  \leq \bigl\lceil \frac{N}{r} \bigr\rceil$.
		Let $X_{N,{\bf k}_\central}$ be a random variable with probability generating function $\E{}{q^{X_{N,{\bf k}_\central}}} = \binom{N}{{\bf k}_\central}^{-1} \cdot \qbinom{N}{{\bf k}_\central}_q$.
		Then, as $r$ is fixed and $N \rightarrow \infty$,
		\begin{align}
			\label{asymptotic first moment central string}
			\E{}{X_{N,{\bf k}_\central}}	& \sim \frac{r-1}{4r} N^2 , \\
			\label{asymptotic second moment central string}
			\Var(X_{N,{\bf k}_\central}) & \sim \frac{r^2 -1}{36 r^2} N^3 .
		\end{align}
		Furthermore, for ${\bf k}_\central$ as above and any sequence ${\bf s} = (s_1^{(N)} , \ldots , s_r^{(N)}) \in \mathbf{N}^r$ such that $s^{(N)}_1 + \cdots + s^{(N)}_r = N$ and $\lVert {\bf s} - {\bf k_\central} \rVert = \LandauO(\sqrt{N})$ we have, as $r$ is fixed and $N \rightarrow \infty$,
		\begin{align}
			\label{first moment inner tube is central}
			\E{}{X_{N,{\bf s}}} & \sim \E{}{X_{N,{\bf k}_\central}} , \\
			\label{second moment inner tube is central}
			\Var(X_{N,{\bf s}}) & \sim \Var(X_{N,{\bf k}_\central}) .
		\end{align}	
	\end{prp}
	
	\begin{proof}
		The asymptotic equivalences \eqref{asymptotic first moment central string} and \eqref{asymptotic second moment central string} can be computed from the definition of the elementary symmetric polynomials (recall \autoref{moments q-multinomials}). Recall that we are dealing with sequences of ${\bf k}_\central$'s and ${\bf s}$'s, and that all asymptotics refer to fixed $r$ and $N \rightarrow \infty$. We will treat the first moment for illustration purposes:
		\begin{align*}
			\E{}{X_{N,{\bf k}_\central}}	 & = \frac{e_2({\bf k}_\central)}{2} = \frac{\sum_{i<j} k_i^{(N)} k_j^{(N)}}{2} \sim \frac{\sum_{i<j} \frac Nr \frac Nr}{2} = \frac{\binom{r}{2} \frac{N^2}{r^2}}{2} = \frac{r-1}{4r} N^2 .
		\end{align*}
		Furthermore, for the ${\bf s}$'s in question one has
		\begin{align*}
			e_1({\bf s}) & = e_1({\bf k}_\central) + \LandauO(N^{\frac 12}) , \\
			e_2 ({\bf s}) & = e_2 ({\bf k}_\central) + \LandauO(N^{\frac 32}) , \\
			e_3 ({\bf s}) & = e_3 ({\bf k}_\central) + \LandauO(N^{\frac 52}) .
		\end{align*}
		Therefore, the claimed asymptotic equivalences \eqref{first moment inner tube is central} and \eqref{second moment inner tube is central} follow immediately from \autoref{moments q-multinomials} and the exhibited quadratic and cubic asymptotic growth (in $N$) of $\E{}{X_{N,{\bf k}_\central}}$ and $\Var(X_{N,{\bf k}_\central})$, respectively.
	\end{proof}

	By a method of Panny \cite{MR845446}, we can determine the cumulants of $q$-multinomial coefficients explicitly.
	The same technique has already been applied by Prodinger \cite{MR2019639} to obtain the cumulants of $q$-binomial coefficients.
	The exact formula will be stated as \eqref{exact cumulants of gaussian multinomials}, but in the sequel we will only need the following asymptotic statement:
	
	\begin{lem}
		\label{multinomials cumulant bounds}
		Let $\mathbf{k} = \mathbf{k}^{(N)} \in \mathbf{N}^r$ be any sequence such that $k_1^{(N)} + \cdots + k_r^{(N)} = N$.
		For each $N \in \mathbf{N}$, let $X_{N,\mathbf k}$ be a random variable with probability generating function
		$
			\E{}{q^{X_{N,{\bf k}}}}
			= \binom{N}{{\bf k}}^{-1} \cdot \qbinom{N}{{\bf k}}_q
		$.
		Then, 	for all $j \geq 1$, the $j$-th cumulant of $X_{N,{\bf k}}$ is of order
		\begin{align}
		\label{higher cumulants any}
			\kappa_j(X_{N,{\bf k}})
			= O(N^{j+1})
		\end{align}
	for $N \rightarrow \infty$.
	Furthermore, if
		$
			\lVert {\bf k} - {\bf k}_\central \rVert
			= \LandauO(\sqrt{N})
		$
		for $\mathbf{k}_\mathrm{c}$ as above, then	
		as $r$ and $\alpha$ are fixed, and $N \rightarrow \infty$,
		\begin{align}
			\label{higher moments inner tube are central}
			\E{}{X_{N,{\bf k}}^{\alpha}}
			= \left(\frac{r-1}{4r} N^2 \right)^\alpha + O(N^{2\alpha-1})
			.
		\end{align}
	\end{lem}

	\begin{proof}
		For $k \geq 1$, let $Y_k$ be a random variable with probability generating function $\E{}{q^{Y_k}} = \frac{[k]_q!}{k!}$.
		Denote the $j$-th cumulant of $Y_k$ by $\kappa_{j,k}$.
		Panny \cite[bottom of p.\ 176]{MR845446} shows that
		\[
			\kappa_{j,k}
			= \begin{cases}
				\frac{k(k-1)}{4} & j = 1, \\
				\frac{B_j}{j} \cdot \left( \frac{B_{j+1}(k+1)}{j+1} - k \right) & j \geq 2,
			\end{cases}
		\]
		where $B_j(x)$ denotes the $j$-th Bernoulli polynomial evaluated at $x$ and $B_j$ denotes the $j$-th Bernoulli number.
		Note that $B_j = 0$ for odd $j \geq 3$.
		
	  Our random variable $X_{N,\mathbf k}$ has probability generating function
	  \[
		  	\E{}{q^{X_{N,\mathbf k}}}
			= \binom{N}{{\bf k}}^{-1} \cdot \qbinom{N}{{\bf k}}_q
			= \frac{\frac{[N]_q!}{N!}}{\frac{[k_1]_q!}{k_1!} \cdots \frac{[k_r]_q!}{k_r!}}
			.
	  \]
	  Hence, its cumulant generating function is
	  \[
	  	  \log \E{}{e^{tX_{N,\mathbf k}}}
		  = \log \frac{[N]_{e^t}}{N!} - \log \frac{[k_1]_{e^t}}{k_1!} - \cdots - \log \frac{[k_r]_{e^t}}{k_r!}
	  \]
	  and for $j \geq 2$ its cumulants are
	  \begin{equation} \label{exact cumulants of gaussian multinomials}
	  		\begin{aligned}
		  		\kappa_j(X_{N,\mathbf{k}})
				&= \kappa_{j,N} - \kappa_{j,k_1} - \cdots  - \kappa_{j,k_r} \\
		   	&= \frac{B_j}{j(j + 1)} \Bigl( \begin{aligned}[t]
						& B_{j + 1}(N + 1) \\
						& - B_{j + 1}(k_1 + 1) - \cdots - B_{j + 1}(k_r + 1) \Bigr)
						.
					\end{aligned}
			\end{aligned}
		\end{equation}
	  Since $B_j(x) \sim x^j$ for $x \to \infty$ and each $k_i = O(N)$, and by a similar argument for $j=1$, the first part of the lemma follows.
	  
	  For the second part, suppose $\lVert {\bf k} - {\bf k}_\central \rVert= \LandauO(\sqrt{N})$. Since $k_i^{(N)} = \frac Nr + O(\sqrt N)$ it follows that
	  $
	  		B_j(k_i^{(N)} + 1)
			\sim \bigl( \frac Nr + O(\sqrt N) \bigr)^j
			\sim \frac{1}{r^j} N^j
		$.
	  Hence, if we consider the nonvanishing even cumulants $\kappa_{2\beta}(X_{N,\mathbf{k}})$, we have
	  \begin{align*}
	  		B_{2\beta + 1}(N + 1) - \sum_{i=1}^r B_{2\beta + 1}(k_i^{(N)} + 1)
			&\sim N^{2\beta + 1} - r \frac{1}{r^{2\beta + 1}} N^{2\beta + 1} \\
			&= \underbrace{\left( 1 - \tfrac{1}{r^{2\beta}} \right)}_{\neq 0} N^{2\beta + 1} ,
	  \end{align*}
	  i.e., $\kappa_{2\beta}(X_{N,\mathbf{k}})$ is exactly of order $2\beta + 1$ in $N$.
	 	
		We abbreviate $\kappa_j(X_{N,\mathbf{k}})$ by $\kappa_j$ for convenience.
	 	Recall that the moment generating function is the exponential of the cumulant generating function.
	 	Consequently, one has the following standard relation between higher moments and cumulants:
	 	\begin{equation} \label{moments from cumulants}
	 		\E{}{X_{N,{\bf k}}^\alpha}
			= \sum_{\substack{\pi_1 + 2\pi_2 + \cdots + \alpha \pi_\alpha = \alpha \\ \pi_i \in \{ 0,1,\ldots , \alpha\}}}
			\left( \frac{\kappa_1}{1!} \right)^{\pi_1}
			\cdots \left( \frac{\kappa_\alpha}{\alpha!} \right)^{\pi_\alpha} \frac{\alpha!}{\pi_1!
			\cdots \pi_\alpha!}
			.
	 	\end{equation}
	 	Since the $j$-th Bernoulli polynomial has degree $j$ and the higher cumulants $\kappa_j$ of odd order vanish, we have (cf.\ \cite[\S 3]{MR845446})
	 	\begin{align*}
	 		\left( \frac{\kappa_1}{1!} \right)^{\pi_1} \left( \frac{\kappa_2}{2!} \right)^{\pi_2} \cdots \left( \frac{\kappa_\alpha}{\alpha!} \right)^{\pi_\alpha} &= O(N^{2\pi_1 + 3\pi_2 + \cdots + (\alpha +1)\pi_\alpha}) \\
			&= O(N^{\alpha + \pi_1 + \pi_2 +\cdots + \pi_\alpha}) .
	 	\end{align*}
	 	This leads to the asymptotic expansion
		\begin{equation} \label{asymtotics of higher moments of strings}
	 		\E{}{X_{N,{\bf k}}^{\alpha}} = \kappa_1(X_{N,{\bf k}})^\alpha + O(N^{2\alpha-1})
			.
		\end{equation}
	 	But
	 	$
	 		\kappa_1(X_{N,{\bf k}})= \E{}{X_{N,{\bf k}}}$, and since $|| {\bf k} - {\bf k}_\central || = O(\sqrt{N})
		$,
		our \autoref{asymptotic moments central q-multinomials} yields
	 	\[
			\kappa_1(X_{N,{\bf k}}) \sim \frac{r-1}{4r} N^2
		\]
		as $N \rightarrow \infty$.
		This finishes the proof.
	\end{proof}
	
	For our proof of \autoref{gauss} by the method of moments, we need to show that the moments of the standardized central $q$-multinomial coefficients converge to the moments of the standard normal distribution.
	We will show this more generally for $q$-multinomial coefficients in $O(\sqrt{N})$-distance to the center, since our arguments yield this without extra effort. Similar results have been obtained by Canfield, Janson and Zeilberger \cite{cjz11} (for a history concerning this distribution see their erratum).
	
	\begin{prp} \label{CJZ}
	Let ${\bf k}_\central = (k_1^{(N)} , \ldots , k_r^{(N)}) \in \mathbf{N}^r$ be a sequence such that $k_1^{(N)} + \cdots + k_r^{(N)} = N$ and $\bigl\lfloor \frac{N}{r} \bigr\rfloor \leq k_i^{(N)}  \leq \bigl\lceil \frac{N}{r} \bigr\rceil$.
	Furthermore, consider a sequence ${\bf s} = (s_1^{(N)} , \ldots , s_r^{(N)}) \in \mathbf{N}^r$ such that $s^{(N)}_1 + \cdots + s^{(N)}_r = N$ and $\lVert {\bf s} - {\bf k_\central} \rVert = \LandauO(\sqrt{N})$.
	Let $X_{N,\mathbf s}$ be a random variable with probability generating function
		$
			\E{}{q^{X_{N,{\bf s}}}}
			= \binom{N}{{\bf s}}^{-1} \cdot \qbinom{N}{{\bf s}}_q
		$.
		Then the moments of the random variable
		\[
			\tilde{X}_{N,\mathbf{s}}
			= \frac{X_{N,\mathbf{s}} - \E{}{X_{N,\mathbf{s}}}}{\sqrt{\Var{X_{N,\mathbf{s}}}}}
		\]
		converge to the moments of the standard normal distribution.
		In particular, the distribution of $\tilde{X}_{N,\mathbf{s}}$ converges weakly to the standard normal distribution.
	\end{prp}
	
	\begin{proof}
		Once we have shown the the convergence of the moments, the weak convergence of the distributions follows by the method of moments.
		Hence we will concentrate on showing the convergence of the moments.
		Since the moments of a random variable depend polynomially, in particular continuously, on its cumulants, it is sufficient to show that the cumulants of $\tilde{X}_{N,\mathbf{s}}$ converge to the cumulants of the standard normal distribution, $0, 1, 0, 0, 0, \ldots$
		It follows directly from the definition of cumulants that
		\begin{align*}
			\kappa_j(X + c)
			&= \begin{cases} \kappa_j(X) + c & \text{if } j = 1, \\ \kappa_j(X) & \text{if }j \geq 2, \end{cases}
			\\
			\kappa_j(cX)
			&= c^j \kappa_j(X)
			.
		\end{align*}
		Hence, for $j \geq 3$,
		\[
			\kappa_j(\tilde X_{N,\mathbf s})
			= \frac{\kappa_j(X_{N,\mathbf s})}{\Var(X_{N,\mathbf s})^{j/2}} \\
			= O(N^{j + 1 - 3j/2})
			\to 0
		\]
		by \eqref{asymptotic second moment central string}, \eqref{second moment inner tube is central} and \eqref{higher cumulants any}.
	\end{proof}

	We are ready to prove the advertised asymptotic normality.
	
	\begin{thm}
		\label{gauss}
		For $r\geq2$ and $N \in \mathbf{N}$ let $G_{N,r}$ be a random variable with probability generating function $\E{}{q^{G_{N,r}}} = r^{-N} \cdot G_N^{(r)} (q)$.
		Then,
		\begin{align*}
			\E{}{G_{N,r}} & = \frac{r-1}{4r} N(N-1) , \\ 
			\Var(G_{N,r}) & = \frac{(r-1)(r+1)}{72 r^2} N(N-1)(2N+5) . 
		\end{align*}
		For fixed $r$ and $N \rightarrow \infty$, the distribution of the random variable
		\[
			\frac{G_{N,r} -\E{}{G_{N,r}}}{\Var(G_{N,r})^{\frac 12}}
		\]
		converges weakly to the standard normal distribution.
	\end{thm}
	
	\begin{proof}
		We will deploy the method of moments (cf.\ \cite[Theorem 4.5.5]{MR1796326}). All asymptotics refer to $N \rightarrow \infty$. We will show that for any $\alpha$,
		\begin{equation}
		\label{eq:higher-moments-asymptotically-equivalent}
			\E{}{G_{N,r}^\alpha}
			\sim \E{}{X_{N,\mathbf{k}_\mathrm{c}}^\alpha}	
			.
		\end{equation}

		Note that once we have shown \eqref{eq:higher-moments-asymptotically-equivalent}, the theorem follows by the method of moments: We have already shown in \autoref{CJZ} that the moments of the standardized $X_{N,\mathbf{k}_\mathrm{c}}$ converge to the moments of the standard normal distribution, hence the same holds for the standardized $G_{N,r}$ by linearity of the expected value and the binomial theorem.
		
		We use the notation
		\[
			f(N) \lesssim g(N)
			\quad \text{if} \quad
			\limsup_{N \to \infty} \frac {f(N)}{g(N)} \leq 1
			.
		\]
    In order to verify \eqref{eq:higher-moments-asymptotically-equivalent}, we will show 
    	$
			\E{}{G_{N,r}^\alpha}
			\gtrsim \E{}{X_{N,\mathbf{k}_\mathrm{c}}^\alpha}	
		$
		and
    	$
			\E{}{G_{N,r}^\alpha}
			\lesssim \E{}{X_{N,\mathbf{k}_\mathrm{c}}^\alpha}	
		$
    separately.
    
    Let us start with
    	$
			\E{}{G_{N,r}^\alpha}
			\gtrsim \E{}{X_{N,\mathbf{k}_\mathrm{c}}^\alpha}	
		$.
    For this, it is sufficient to prove that for all $\epsilon > 0$ there is an $N_0 \in \mathbf{N}$ such that
		\begin{equation}
		\label{eq:higher-moments-epsilon-formulation}
			\E{}{G_{N,r}^\alpha}
			\geq
			(1 - \epsilon) \E{}{X_{N,{\bf k}_\central}^\alpha}
		\end{equation}
		for all $N \geq N_0$.
		Let $\epsilon > 0$.
		For $N \in \mathbf{N}$, let
		$
			U_N \subset \{ \mathbf{x} \in \mathbf{R}^r : x_1 + \cdots + x_r = N \}
		$
		be the ball around $(\frac Nr, \ldots, \tfrac Nr)$ of minimal radius such that
		$
			\sum_{\mathbf{k} \in U_N} \binom{N}{\mathbf{k}}
			\geq
			\sqrt{1-\epsilon} \cdot r^N
		$.
		By the central limit theorem for ordinary multinomial coefficients, the radii of $U_N$ are proportional to $\sqrt N$ up to an error of order $O(1)$.	
		Choose a sequence $\mathbf{k}_{\min}$ with $\mathbf{k}_{\min}^{(N)} \in U_N$ such that $\E{}{X_{N,\mathbf{k}_{\min}}^\alpha} = \min_{\mathrm{k} \in U_N} \E{}{X_{N,\mathbf{k}}^\alpha}$.
		By \eqref{higher moments inner tube are central}, $\E{}{X_{N,\mathbf{k}_{\min}}^\alpha} \sim \E{}{X_{N,\mathbf{k}_{\central}}^\alpha}$.
		Hence there is an $N_0$ such that
		\[
			\E{}{X_{N,\mathbf{k}_{\min}}^\alpha}
			\geq \sqrt{1-\epsilon} \cdot \E{}{X_{N,\mathbf{k}_{\central}}^\alpha}
		\]
		for all $N \geq N_0$.
		Consequently
		\begin{align*}
			\E{}{G_{N,r}^\alpha}
			&= r^{-N} \sum_{\mathbf{k}} \binom{N}{\mathbf{k}} \E{}{X_{N,\mathbf{k}}^\alpha} \\
			&\geq r^{-N} \sum_{\mathbf{k} \in U_N} \binom{N}{\mathbf{k}} \E{}{X_{N,\mathbf{k}}^\alpha} \\
			&\geq \underbrace{r^{-N} \sum_{\mathbf{k} \in U_N} \binom{N}{\mathbf{k}}}_{\geq \sqrt{1-\epsilon}} \underbrace{\E{}{X_{N,\mathbf{k_{min}}}^\alpha}}_{\geq \sqrt{1-\epsilon} \cdot \E{}{X_{N,\mathbf{k_\mathrm{c}}}^\alpha}} \\
			&\geq (1 - \epsilon) \E{}{X_{N,\mathbf{k_\mathrm{c}}}^\alpha} .
		\end{align*}
		
		We continue by showing
		$
			\E{}{G_{N,r}^\alpha}
			\lesssim \E{}{X_{N,\mathbf{k}_\mathrm{c}}^\alpha}	
		$.
		We claim that for all sequences $\mathbf{k} = \mathbf{k}^{(N)}$,
		\begin{equation} \label{asymptotically centrally bounded}
			\E{}{X_{N,\mathbf{k}}^\alpha}
			\lesssim \E{}{X_{N,\mathbf{k_\mathrm{c}}}^\alpha}
			.
		\end{equation}
		In order to show this, we will treat the summands of the expression \eqref{moments from cumulants} for $\E{}{X_{N,\mathbf{k}}^\alpha}$ individually.
		Let us start with the index $\pi = (\alpha, 0, \ldots, 0)$, i.e., the summand $\kappa_1(X_{N,\mathbf{k}})^\alpha$.
		Note that
		\begin{equation} \label{max of e2 on simplex}
			\max_{\mathbf{x} \in N\Delta_{r-1}} e_2(\mathbf{x})
			= e_2 \left( \tfrac Nr, \ldots, \tfrac Nr \right)
			= \frac{r-1}{2r} N^2
			.
		\end{equation}
		Here $N\Delta_{r-1}$ denotes the dilated $(r-1)$-dimensional standard simplex in $\mathbf{R}^r$.
		Hence,
		\[
			\kappa_1(X_{N,\mathbf{k}})
			= \frac{e_2(\mathbf{k})}{2} \leq \frac{r-1}{4r} N^2
		\]
		for all $N$, so our summand is bounded from above by
		\[
			\kappa_1(X_{N,\mathbf{k}})^\alpha
			\leq \left( \frac{r-1}{4r} N^2 \right)^\alpha
			.
		\]
		Now, consider a summand of \eqref{moments from cumulants} with index $\pi \neq (\alpha, 0, \ldots, 0)$.
		By the same arguments that lead to \eqref{asymtotics of higher moments of strings} such a summand is $O(N^{2\alpha - 1})$.
		Since the number of summands (the number of partitions of $\alpha$) does not depend on $N$, we have
		\begin{align*}
			\E{}{X_{N,\mathbf{k}}^\alpha}
			&\leq \left( \frac{r-1}{4r} N^2 \right)^\alpha + O(N^{2\alpha - 1}) \\
			&\lesssim \left( \frac{r-1}{4r} N^2 \right)^\alpha \\
			&\annrel{\eqref{higher moments inner tube are central}}{\sim} \E{}{X_{N,\mathbf{k}_\mathrm{c}}^\alpha}
			.
		\end{align*}
		This proves \eqref{asymptotically centrally bounded}.
		We are ready to prove the final inequality $\E{}{G_{N,r}^\alpha}
			\lesssim \E{}{X_{N,\mathbf{k}_\mathrm{c}}^\alpha}	$:
		Choose a sequence $\mathbf{k}_{\max} = \mathbf{k}_{\max}^{(N)}$ in $\mathbf{N}^r$ such that
		$
			\E{}{X_{N,\mathbf{k}_{\max}}^\alpha}
			= \max_{\mathbf{k}} \E{}{X_{N,\mathbf{k}}^\alpha}
		$.
		Then
		\begin{align*}
			\E{}{G_{N,r}^\alpha}
			&= \frac{1}{r^N} \sum_\mathbf{k} \binom{N}{\mathbf{k}} \E{}{X_{N,\mathbf{k}}^\alpha} \\
			&\leq \frac{1}{r^N} \sum_\mathbf{k} \binom{N}{\mathbf{k}} \E{}{X_{N,\mathbf{k_{\max}}}^\alpha} \\
			&= \E{}{X_{N,\mathbf{k_{\max}}}^\alpha} \\
			&\annrel{\eqref{asymptotically centrally bounded}}{\lesssim} \E{}{X_{N,\mathbf{k_\mathrm{c}}}^\alpha}
			.
		\end{align*}
		This verifies \eqref{eq:higher-moments-asymptotically-equivalent}, and finishes the proof of the theorem.
	\end{proof}

	\section{Galois numbers and inversion statistics}
	\label{sec:macmahon}
	
		Let $\mathcal{P}_{N,r}$ denote the set of partitions of an arbitrary integer into up to $r$ parts of size up to $N$.
		Let $p_{N,r} = \lvert \mathcal{P}_{N,r} \rvert$ denote the number of such partitions.
		Then
		\begin{equation} \label{Drwnq9F6}
			p_{N,r} = \binom{N+r}{N} .
		\end{equation}
		For $T \subset \{1, \ldots, N\}$ with $t := \lvert T \rvert \leq r$, let
		\[
			\mathcal{P}_{N,r}^T
			=
			\{ \lambda \in \mathcal{P}_{N,r} : \{ \lambda_1, \ldots, \lambda_r \} \supset T \}
		\]
		denote the subset of $\mathcal{P}_{N,r}$ consisting of all the partitions such that each element of $T$ occurs as the size of a part.
		Then removing one part of each size in $T$ defines a bijection $\mathcal{P}_{N,r}^T \to \mathcal{P}_{N,r-t}$.
		Hence
		\begin{equation} \label{fZZMo8yc}
			\lvert \mathcal{P}_{N,r}^T \rvert
			= p_{N,r-t}
			\stackrel{\eqref{Drwnq9F6}}{=} \binom{N+r-t}{N}.
		\end{equation}
		Note that we adhere to a more restrictive definition of binomial coefficients, namely $\binom{a}{b} = 0$ unless $0 \leq b \leq a$.
		 
		\begin{thm}
			\label{macmahon}
			Consider the Galois number $G_N^{(r)}(q) \in \mathbf{N}[q]$ for $r\geq2$ and $N \in \mathbf{N}$. Let $\mathfrak{S}_N$ be the symmetric group on $N$ elements, and for a permutation $\pi \in \mathfrak{S}_N$ denote by $\inv(\pi)$ its number of inversions and by $\des(\pi)$ the cardinality of its descent set $D(\pi)$. Then,
			\[
				G_N^{(r)}(q) = \sum_{\pi \in \mathfrak{S}_N}\binom{N+r-1-\des(\pi)}{N}  q^{\inv(\pi)}
				.
			\]
			For fixed $N$ and $r \rightarrow \infty$ we have
			\[
				\frac{N !}{r^N} \cdot G_N^{(r)}(q) \to \sum_{\pi \in \mathfrak{S}_N} q^{\inv(\pi)}  = [N]_q!
				.
			\]			
		\end{thm}
		
		\begin{proof}
			$G_N^{(r)}(q) = H^{(r)}_N(\mathbf{1}, q)$, and by definition
			\[
				H^{(r)}_N(\mathbf{1}, q)
				= \sum_{\substack{\mathbf{k} \in \mathbf{N}^r \\ k_1 + \cdots + k_r = N}} \qbinom{N}{\mathbf{k}}_q
				.
			\]
			Note that there is a bijection $\mathcal{P}_{N,r-1} \to \{ \mathbf{k} \in \mathbf{N}^r : k_1 + \cdots + k_r = N \}$ given by
			\[
				\mathbf{k}_\lambda
				:=
				(N - \lambda_1, \lambda_1 - \lambda_2, \ldots, \lambda_{r-2}-\lambda_{r-1}, \lambda_{r-1})
				.
			\]
			Hence
			\[
				H^{(r)}_N(\mathbf{1}, q)
				= \sum_{\lambda \in \mathcal{P}_{N,r-1}} \qbinom{N}{\mathbf{k}_\lambda}_q
				.
			\]
			By \cite[Chapter 2, (20)]{MR1442260}
			\[
				\qbinom{N}{{\mathbf{k}_\lambda}}_q
				= \sum_{\substack{\pi \in \mathfrak{S}_N \\ D(\pi) \subset \{\lambda_1, \ldots, \lambda_{r-1}\}}} q^{\inv(\pi)} .
			\]
			Hence 	
			\begin{align*}
				H^{(r)}_N(\mathbf{1}, q)
				&=
				\sum_{\lambda \in \mathcal{P}_{N,r-1}}
				\sum_{\substack{\pi \in \mathfrak{S}_N \\ D(\pi) \subset \{\lambda_1, \ldots, \lambda_{r-1}\}}}
				q^{\inv(\pi)}
				\\ &=
				\sum_{\lambda \in \mathcal{P}_{N,r-1}}
				\sum_{T \subset \{\lambda_1, \ldots, \lambda_{r-1}\}}
				\sum_{\substack{\pi \in \mathfrak{S}_N \\ D(\pi) = T}}
				q^{\inv(\pi)}
				\\	 &=
				\sum_{T \subset \{1, \ldots, N-1\}}
				\lvert \mathcal{P}_{N,r-1}^T \rvert
				\sum_{\substack{\pi \in \mathfrak{S}_N \\ D(\pi) = T}}
				q^{\inv(\pi)} \\
				& = \sum_{t=0}^{N-1} \binom{N+r-1-t}{N} \sum_{\substack{\pi \in \mathfrak{S}_N \\ \des(\pi) = t}} q^{\inv(\pi)}
.
			\end{align*}
			The last equality follows from \eqref{fZZMo8yc} which implies our first assertion. As for the second claim, note that as $N$ is fixed, $0 \leq t \leq N-1$  and $r \rightarrow \infty$ we have
			\[
				\binom{N+r-1-t}{N} \sim \frac{r^N}{N!}
				.
			\]
			The stated decomposition of the inversion statistic as a $q$-factorial is well-known.
		\end{proof}
		
	\section{Applications and discussions}
	\label{sec:applications}
		
		\subsection{Rogers--Szeg\H{o} polynomials}
			Our \autoref{mean variance galois} is a sufficient ingredient to determine the covariance of the overall distribution of coefficients in the generalized $N$-th Rogers--Szeg\H{o} polynomial.
			
			\begin{cor}
			\label{covariance Rogers--Szegoe}
				Given $N, r > 0$, let $(\mathbf X; Y) = (X_1, \ldots ,X_r;Y)$ be a random vector with probability generating function $\E{}{\mathbf z ^\mathbf X q^Y} = r^{-N}H^{(r)}_N({\bf z},q)$. Then, the covariance of $(\mathbf{X}; Y)$ is given by
				\begin{align*}
					\Sigma(\mathbf X; Y) & =
					\begin{pmatrix}
						\Sigma(\mathbf X) & 0 \\
						0 & \Var (G_{N,r})
					\end{pmatrix}
					,
				\end{align*}
				where $\Sigma(\mathbf X)$ is the covariance of the multinomial distribution.
			\end{cor}
			
			\begin{proof}
				The diagonal (block) entries are clear, since the specialization at $({\bf z},1)$ gives the multinomial distribution whereas the specialization at $({\bf 1},q)$ is exactly the generalized Galois number studied in \autoref{mean variance galois}.
				Therefore, we only have to prove that $\Cov(X_i,Y) = 0$ for $i=1,\ldots,r$, which can be shown as follows:
				By symmetry, it is clear that all $\Cov(X_i,Y)$  coincide.
				As $X_1 + \cdots + X_r = N$ almost surely, it follows that
				\[
					0 = \Cov(N,Y) = N \Cov(X_1, Y) . \qedhere
				\]
			\end{proof}
		
		\subsection{Linear $q$-ary codes}
			Consider the classical Galois numbers $G_n(q) = G_n^{(2)}(q)$ that count the number of subspaces of $\mathbf{F}_q^n$.
			For a general prime power $q$, Hou \cite{MR2177491,MR2492098} studies the number of equivalence classes $N_{n,q}$ of linear $q$-ary codes of length $n$ under three notions of equivalence: permutation equivalence ($\mathfrak{S}$), monomial equivalence ($\mathfrak{M}$), and semi-linear monomial equivalence ($\Gamma$).
		He proves
		\begin{align*}
			N_{n,q}^{\mathfrak{S}} & \sim \frac{G_n(q)}{n!} , \\
			N_{n,q}^{\mathfrak{M}} & \sim \frac{G_n(q)}{n!(q-1)^{n-1}} , \\
			N_{n,q}^{\Gamma} & \sim \frac{G_n(q)}{n!(q-1)^{n-1}a} ,
		\end{align*}
		where $a = \left| \mathrm{Aut}(\mathbf{F}_q) \right| = \log_p(q)$ with $p = \mathrm{char}(\mathbf{F}_q)$.
		The case of binary codes up to monomial equivalence, $N_{n,2}^{\mathfrak{S}} \sim \frac{G_n(2)}{n!}$, is previously derived by Wild \cite{MR1755766, MR2191288}.
		Now, the following corollary is immediate from our \autoref{macmahon}.
		
		\begin{cor}
		\label{binary codes and matroids}
			The number of linear $q$-ary codes of length $n$ up to equivalence $(\mathfrak{S})$, $(\mathfrak{M})$ and $(\Gamma)$ is given asymptotically, as $q$ is fixed and $n \rightarrow \infty$, by
			 \begin{align}
			 	N_{n,q}^{\mathfrak{S}} & \sim \frac{1}{n!} \sum_{\substack{\pi \in \mathfrak{S}_n \\ \des(\pi) \leq 1}}  \binom{n+1-\des(\pi)}{n} q^{\inv(\pi)} , \\
				N_{n,q}^{\mathfrak{M}} & \sim \frac{1}{n!(q-1)^{n-1}} \sum_{\substack{\pi \in \mathfrak{S}_n \\ \des(\pi) \leq 1}}  \binom{n+1-\des(\pi)}{n} q^{\inv(\pi)}, \\
				N_{n,q}^{\Gamma} & \sim \frac{1}{n!(q-1)^{n-1}a} \sum_{\substack{\pi \in \mathfrak{S}_n \\ \des(\pi) \leq 1}}  \binom{n+1-\des(\pi)}{n} q^{\inv(\pi)} ,
			\end{align}
			where $a = \left| \mathrm{Aut}(\mathbf{F}_q) \right| = \log_p(q)$ with $p = \mathrm{char}(\mathbf{F}_q)$.
			In particular, the numerator of the asymptotic numbers of linear $q$-ary codes is the (weighted) inversion statistic on the permutations having at most $1$ descent. 
		\end{cor}
		
		\subsection{The symmetric group acting on subspaces over finite fields}

		Consider the character $\chi_N(\tau) = \# \{ V \subset \mathbf{F}_q^N : \tau V = V \}$ of the symmetric group $\mathfrak{S}_N$ acting on $\mathbf{F}_q^N$ by permutation of coordinates.
		Lax \cite{MR2067601} shows that the normalized character $\chi_N(\tau)/G_N(q)$ asymptotically approaches the character which takes the value $1$ on the identity and $0$ otherwise.
		
		\begin{cor}
			Consider the character $\chi_N (\tau)$ of the symmetric group $\mathfrak{S}_N$ acting on $\mathbf{F}_q^N$ by permutation of coordinates. Then, as $N \rightarrow \infty$, the normalized character
			\begin{align}
				\frac{\chi_N (\tau)}{\sum_{\substack{\pi \in \mathfrak{S}_N \\ \des(\pi) \leq 1}}  \binom{N+1-\des(\pi)}{N} q^{\inv(\pi)}} \rightarrow \delta_{\mathrm{id},\tau} .
			\end{align}
			In particular, the character $\chi_N$ approaches asymptotically the (weighted) inversion statistic on the permutations having at most $1$ descent.
		\end{cor}

		\subsection{Affine Demazure modules}
			
			We refer the reader to \cite{MR2188930,MR2235341,MR1104219} for the basic facts about the representation theory of affine Kac-Moody algebras and Demazure modules, and the notation used.
			Now, according to \cite[Equation (3.4)]{MR1457389} and \cite[Theorem 6 and 7]{MR1771615}, certain Demazure characters can be described via generalized Rogers--Szeg\H{o} polynomials.
		
		\begin{lem}
		\label{demazure character as Rogers--Szegoe}
			Let $r,N \in \mathbf{N}$, $r \geq 2$ and $0 \leq i < r$, $i \equiv N \mod r$.
			Let $q = e^{\delta}$, ${\bf z} = (e^{\Lambda_1 - \Lambda_0}, e^{\Lambda_2 -\Lambda_1}, \ldots , e^{\Lambda_{r-1} - \Lambda_{r-2}}, e^{\Lambda_0 -\Lambda_{r-1}})$, and
			\begin{align*}
				d_r(N) & = \frac{N(N-1)}{2} - \frac{(N-i)(N+i-r)}{2r} .
			\end{align*}
			Then, the character of the $\widehat{\mathfrak{sl}}_r$ Demazure module $V_{-N \omega_1}(\Lambda_0)$ is given by
			\[
				\charac{V_{-N \omega_1}(\Lambda_0)} = e^{\Lambda_0 - d_r(N)\delta} \cdot H_N^{(r)}({\bf z} , q) \quad \in \mathbf{N}[{\bf z} , q^{-1}]
				.
			\]
		\end{lem}
		
		\begin{proof}
			Note that $t_{-\omega_1} = s_1 s_2 \ldots s_{r-1} \sigma^{r-1}$	with $\sigma$ being the automorphism of the Dynkin diagram of $\widehat{\mathfrak{sl}}_r$ which sends $0$ to $1$ (see e.g.~\cite[\S 2]{MR1976581}). Furthermore, following \cite[\S 2]{MR1771615} we have $[N] = t_{-N\omega_1} \cdot \eta_N$ with the convention $\sigma \cdot \eta_N = \eta_N$. Here, $[N] = (N,0,\ldots,0) \in \mathbf{N}^r$ denotes the one-row Young diagram, and $\eta_N$ the smallest composition of degree $N$. That is, when $N = kr +i$ with $0\leq i<r$, we have $\eta_N = ((k)^{r-i},(k+1)^i) \in \mathbf{N}^r$.
			Then, by \cite[Theorem 6 and 7]{MR1771615}\footnote{There seems to be a missprint in \cite[\S 4]{MR1771615}. Namely, the image $\pi(q)$ should equal $q = e^\delta$, not $q = e^{-\delta}$.}  and \cite[Equation (3.4)]{MR1457389} we have
			\[
				\charac{V_{-N \omega_1}(\Lambda_0)} = e^{\Lambda_0 - d_r(N)\delta} \cdot P_{[N]}({\bf z};q,0) = e^{\Lambda_0 - d_r(N)\delta} \cdot H_N^{(r)}({\bf z} , q)
				,
			\]
			where $P_{[N]}({\bf z};q,0)$ denotes the specialized symmetric Macdonald polynomial (see \cite[Chapter VI]{MR1354144}) associated to the partition $[N]$.
		\end{proof}
		
		\begin{exl}
			For $r=2$ consider the specialization of the Rogers--Szeg\H{o} polynomial $H_4^{(2)}(z,z^{-1},q)$ and the Demazure module  $V_{-4 \omega_1}(\Lambda_0)$. Via Demazure's character formula we obtain
			\begin{align*}
				\charac{V_{-4 \omega_1}(\Lambda_0)}
					& = (e^{4 (\Lambda_1 - \Lambda_0)}+e^{-4 (\Lambda_1 - \Lambda_0)})e^{-4\delta}\\
					 & \quad +(e^{2 (\Lambda_1 - \Lambda_0)} + e^{-2 (\Lambda_1 - \Lambda_0)})(e^{-\delta}+e^{-2\delta}+e^{-3\delta}+e^{-4\delta}) \\
					 & \quad+ e^{0 (\Lambda_1 - \Lambda_0)}(e^{0\delta} + e^{-\delta}+2e^{-2\delta}+e^{-3\delta}+e^{-4\delta})
				.
			\end{align*}
			Furthermore, by definition
			\begin{align*}
				H_4^{(2)}(z,z^{-1},q)
					& = (z^4 + z^{-4})q^0  +(z^2 + z^{-2})(q^3 + q^2 +q +q^0) \\
					& \quad +z^0 (q^4+q^3 +2 q^2 +q +q^0)
				.
			\end{align*}
			Hence, with ${\bf z} = (e^{\Lambda_1 - \Lambda_0}, e^{\Lambda_0 - \Lambda_1})$, $q=e^\delta$ and $d_2 (4) =4$ we have the equality
			\[
				\charac{V_{-4 \omega_1}(\Lambda_0)} = e^{\Lambda_0 - d_2 (4)\delta} \cdot H_4^{(2)}(z,z^{-1},q)
			\]
			as claimed.
		\end{exl}
		
		The coefficient $l$ in $e^{-l \delta}$ is commonly referred to as the degree of a monomial in $\charac{V_{-N \omega_1}(\Lambda_0)}$.
		When $d$ is a scaling element, the polynomial $\charac{V_{-N \omega_1}(\Lambda_0)} |_{\mathbf{C}d} \in \mathbf{N}[e^\delta]$ is called the basic specialization of the Demazure character (see \cite[\S 1.5, 10.8, and 12.2]{MR1104219} for the terminology in the context of integrable highest weight representations of affine Kac-Moody algebras). Based on the relation described in \autoref{demazure character as Rogers--Szegoe}, we summarize our main results, \autoref{gauss} and \autoref{macmahon}, in this language.
		
		\begin{cor}
		\label{demazure basic specialization}
			For $r \geq 2$ and $N \in \mathbf{N}$ consider the $\widehat{\mathfrak{sl}}_r$ Demazure module $V_{-N \omega_1}(\Lambda_0)$. Let $\Gamma_{N,r}$ be a random variable with probability generating function $\E{}{e^{\Gamma_{N,r}\delta}} = r^{-N} \cdot \charac{V_{-N \omega_1}(\Lambda_0)} |_{\mathbf{C}d} \in \mathbf{Q}[e^\delta]$.
			Then, for $0 \leq i < r$, $i \equiv N \mod r$ we have
			\begin{align}
				\E{}{\Gamma_{N,r}}
				&= \frac{(r+1)N(N-1) - 2(N-i)(N+i-r)}{4r}, \\
				\Var(\Gamma_{N,r})
				&= \frac{(r-1)(r+1)}{72 r^2} N(N-1)(2N+5)
				.
			\end{align}
			For fixed $r$ and $N \rightarrow \infty$, the distribution of the random variable
			\begin{align}
				\frac{\Gamma_{N,r} - \E{}{\Gamma_{N,r}}}{\Var(\Gamma_{N,r})^{1/2}}
			\end{align}
			converges weakly to the standard normal distribution $\mathcal{N}(0,1)$.
			Furthermore, let $\mathfrak{S}_N$ be the symmetric group on $N$ elements, and for a permutation $\pi \in \mathfrak{S}_N$ denote by $\inv(\pi)$ its number of inversions and by $\des(\pi)$ the cardinality of its descent set $D(\pi)$. Then, with $a_{N,r} = N(N-1)/2 - (N-i)(N+i-r)/2r$
			\begin{align}
			\label{demazure weighted inversion statistic}
				\frac{\charac{V_{-N \omega_1}(\Lambda_0)}|_{\mathbf{C}d}}{e^{\Lambda_0 - d_r (N) \delta}} = \sum_{\pi \in \mathfrak{S}_N}\binom{N+r-1-\des(\pi)}{N}  e^{\inv(\pi)\delta}
				.
			\end{align}
			For fixed $N$ and $r \rightarrow \infty$ we have
			\begin{align}
			\label{demazure length statistic}
				\frac{N !}{r^N} \cdot \frac{\charac{V_{-N \omega_1}(\Lambda_0)}|_{\mathbf Cd}}{e^{\Lambda_0 - d_r (N) \delta}} \to \sum_{\pi \in \mathfrak{S}_N} e^{\inv(\pi)\delta}  = [N]_{e^\delta}!
				.
			\end{align}
		\end{cor}
	
			It is interesting to continue the investigation of the basic specialization of Demazure characters including Kac-Moody algebra types different from $A$. The starting point should be Ion's article \cite{MR1953294} which is a generalization of Sanderson's work \cite{MR1771615}, and one should also consider \cite{MR1607188}.
			In view of \eqref{demazure length statistic} and \cite{MR1953294,MR1607188} we propose the following conjecture 

			\begin{cnj}
				Let $X=A,B,D$ and $r \in \mathbf{N}$. Consider the $\widehat{X}_r$ Demazure module $V_{-N \omega_1}(\Lambda_0)$ and let $d_r^X(N)$ be the maximal occuring degree. For fixed $N$ and $r \to \infty$ it holds
				\[
					\frac{\# W(X_N)}{\dim (V(\omega_1)^{\otimes N} )} \cdot \frac{\charac{V_{-N \omega_1}(\Lambda_0)}|_{\mathbf Cd}}{e^{\Lambda_0 - d_r^X (N) \delta}} \to \sum_{w \in W(X_N)} e^{l(w)\delta}
					.
				\]
			Here, $W(X_N)$ is the Weyl group of finite type $X_N$, $l : W(X_N) \to \mathbf{N}$ is the length function, and $V(\omega_1)$ denotes the standard representation of the finite-dimensional Lie algebra of type $X_r$. 
			\end{cnj}
			
			Note that \eqref{demazure length statistic} proves the case $X=A$. It is interesting to investigate in an analogue of \eqref{demazure weighted inversion statistic} for the types $B$ and $D$.
			
		\subsection{Descent-inversion statistics}
		
			Stanley \cite{MR0409206} derived a generating function identity for the joint probability generating function of the descent-inversion statistic on the symmetric group $\mathfrak{S}_N$:
			\[
				\sum_{N=0}^\infty \sum_{\pi \in \mathfrak{S}_N} t^{\des(\pi)}q^{\inv(\pi)} \frac{u^N}{[N]_q !} = \frac{1-t}{\mathrm{Exp}_q (u(t-1))-t}
				,
			\]
			where $\mathrm{Exp}_q (x) = \sum q^{\binom{n}{2}} x^n / [n]_q ! $. Motivated by \autoref{macmahon} we define a weighted joint probability generating function of the descent-inversion statistic on the symmetric group $\mathfrak{S}_N$ by
			\begin{equation}
			\label{deformed galois}
				G_N^{(r)}(q,t) =  \sum_{\pi \in \mathfrak{S}_N} \binom{N + r-1-\des(\pi)}{N}  t^{\des(\pi)} q^{\inv(\pi)}
				.
			\end{equation}
			Note that $G_N^{(r)}(q,1) = G_N^{(r)}(q)$. It is interesting to investigate the generating function
			\begin{equation}
			\label{genfunc deformed galois}
				\sum_{N=0}^\infty G_N^{(r)}(q,t) \frac{u^N}{[N]_q !} ,
			\end{equation}
			possibly with refinements depending on $r$ and $t$, and the $t$-deformed generalized Galois number $G_N^{(r)}(q,t)$ itself.
	
	\section{Acknowledgements}
		
		This work was supported by the Swiss National Science Foundation (grant PP00P2-128455), the National Centre of Competence in Research ``Quantum Science and Technology,'' the German Science Foundation (SFB/TR12, SPP1388, and grants \mbox{CH~843/1-1}, \mbox{CH~843/2-1}), and the Excellence Initiative of the German Federal and State Governments through the Junior Research Group Program within the Institutional Strategy ZUK 43.
		
		The second author would like to thank Matthias Christandl for his kind hospitality at the ETH Zurich, and Ryan Vinroot for helpful conversations.
		
		We are indebted to the referees for pointing out an error in the proof of \autoref{gauss} in an earlier version of this article.
		
	\bibliographystyle{amsplain}
	\bibliography{galois}

\providecommand{\doi}[1]{\href{http://dx.doi.org/#1}{\nolinkurl{doi:#1}}}
  \providecommand{\arxiv}[1]{\href{http://arxiv.org/abs/#1}{\nolinkurl{arXiv:#1}}}
\providecommand{\bysame}{\leavevmode\hbox to3em{\hrulefill}\thinspace}
\providecommand{\MR}{\relax\ifhmode\unskip\space\fi MR }
\providecommand{\MRhref}[2]{%
  \href{http://www.ams.org/mathscinet-getitem?mr=#1}{#2}
}
\providecommand{\href}[2]{#2}
\begin{thebibliography}{10}

\bibitem{cjz11}
Rodney Canfield, Svante Janson, and Doron Zeilberger, \emph{The {M}ahonian
  probability distribution on words is asymptotically normal}, Adv. in Appl.
  Math. \textbf{46} (2011), 109--124, \doi{10.1016/j.aam.2009.10.001}, Erratum
  February 7, 2012:
  \url{http://www2.math.uu.se/~svante/papers/sj239-erratum.pdf}.

\bibitem{MR2188930}
Roger Carter, \emph{Lie algebras of finite and affine type}, Cambridge Studies
  in Advanced Mathematics, no.~96, Cambridge University Press, 2005,
  \doi{10.1017/CBO9780511614910}.

\bibitem{MR1796326}
Kai~Lai Chung, \emph{A course in probability theory}, third ed., Academic
  Press, 2001.

\bibitem{MR2235341}
Ghislain Fourier and Peter Littelmann, \emph{Tensor product structure of affine
  {D}emazure modules and limit constructions}, Nagoya Math. J. \textbf{182}
  (2006), 171--198.

\bibitem{MR0252232}
Jay Goldman and Gian-Carlo Rota, \emph{The number of subspaces of a vector
  space}, Recent progress in combinatorics: proceedings of the third Waterloo
  conference on combinatorics, May 1968 (W.~T. Tutte, ed.), Academic Press,
  1969, pp.~75--83.

\bibitem{MR1457389}
Kazuhiro Hikami, \emph{Representation of the {Y}angian invariant motif and the
  {M}acdonald polynomial}, J. Phys. A \textbf{30} (1997), 2447--2456,
  \doi{10.1088/0305-4470/30/7/023}.

\bibitem{MR2177491}
Xiang-Dong Hou, \emph{On the asymptotic number of non-equivalent {$q$}-ary
  linear codes}, J. Combin. Theory Ser. A \textbf{112} (2005), 337--346,
  \doi{10.1016/j.jcta.2005.08.001}.

\bibitem{MR2492098}
\bysame, \emph{Asymptotic numbers of non-equivalent codes in three notions of
  equivalence}, Linear Multilinear Algebra \textbf{57} (2009), 111--122,
  \doi{10.1080/03081080701539023}.

\bibitem{MR1953294}
Bogdan Ion, \emph{Nonsymmetric {M}acdonald polynomials and {D}emazure
  characters}, Duke Math. J. \textbf{116} (2003), 299--318,
  \doi{10.1215/S0012-7094-03-11624-5}.

\bibitem{MR1104219}
Victor Kac, \emph{Infinite-dimensional {L}ie algebras}, third ed., Cambridge
  University Press, 1990.

\bibitem{MR1607188}
Atsuo Kuniba, Kailash Misra, Masato Okado, Taichiro Takagi, and Jun Uchiyama,
  \emph{Characters of {D}emazure modules and solvable lattice models}, Nuclear
  Phys. B \textbf{510} (1998), 555--576, \doi{10.1016/S0550-3213(97)00685-8}.

\bibitem{MR2067601}
Robert Lax, \emph{On the character of {$S_n$} acting on subspaces of {${\Bbb
  F}^n_q$}}, Finite Fields Appl. \textbf{10} (2004), 315--322,
  \doi{10.1016/j.ffa.2003.09.001}.

\bibitem{MR1354144}
Ian Macdonald, \emph{Symmetric functions and {H}all polynomials}, second ed.,
  Oxford University Press, 1995.

\bibitem{MR1976581}
\bysame, \emph{Affine {H}ecke algebras and orthogonal polynomials}, Cambridge
  Tracts in Mathematics, no. 157, Cambridge University Press, 2003,
  \doi{10.1017/CBO9780511542824}.

\bibitem{MR845446}
Wolfgang Panny, \emph{A note on the higher moments of the expected behavior of
  straight insertion sort}, Inform. Process. Lett. \textbf{22} (1986),
  175--177, \doi{10.1016/0020-0190(86)90023-2}.

\bibitem{MR2019639}
Helmut Prodinger, \emph{On the moments of a distribution defined by the
  {G}aussian polynomials}, J. Statist. Plann. Inference \textbf{119} (2004),
  237--239, \doi{10.1016/S0378-3758(02)00422-6}.

\bibitem{r1893a}
Leonard Rogers, \emph{On a three-fold symmetry in the elements of {H}eine's
  series}, Proc. Lond. Math. Soc. \textbf{24} (1893), 171--179,
  \doi{10.1112/plms/s1-24.1.171}.

\bibitem{r1893b}
\bysame, \emph{On the expansion of certain infinite products}, Proc. Lond.
  Math. Soc. \textbf{24} (1893), 337--352.

\bibitem{MR1771615}
Yasmine Sanderson, \emph{On the connection between {M}acdonald polynomials and
  {D}emazure characters}, J. Algebraic Combin. \textbf{11} (2000), 269--275,
  \doi{10.1023/A:1008786420650}.

\bibitem{MR0409206}
Richard Stanley, \emph{Binomial posets, {M}\"obius inversion, and permutation
  enumeration}, J. Combinatorial Theory Ser. A \textbf{20} (1976), 336--356.

\bibitem{MR1442260}
\bysame, \emph{Enumerative combinatorics}, vol.~1, Cambridge Studies in
  Advanced Mathematics, no.~49, Cambridge University Press, 1997.

\bibitem{s1926}
G{\'a}bor Szeg\H{o}, \emph{Ein {B}eitrag zur {T}heorie der {T}hetafunktionen},
  Sitzungsberichte Preuss. Akad. Wiss. (1926), 242--252.

\bibitem{v10}
Ryan Vinroot, \emph{Multivariate {R}ogers--{S}zeg{\H{o}} polynomials and flags
  in finite vector spaces}, \arxiv{1011.0984}, 2010.

\bibitem{MR1755766}
Marcel Wild, \emph{The asymptotic number of inequivalent binary codes and
  nonisomorphic binary matroids}, Finite Fields Appl. \textbf{6} (2000),
  192--202, \doi{10.1006/ffta.1999.0273}.

\bibitem{MR2191288}
\bysame, \emph{The asymptotic number of binary codes and binary matroids}, SIAM
  J. Discrete Math. \textbf{19} (2005), 691--699,
  \doi{10.1137/S0895480104445538}.

\end{thebibliography}
	
\end{document}